\documentclass[12pt]{amsart}
\usepackage{amssymb}
\usepackage[all]{xy}
\usepackage{tikz}
\usepackage{amsmath}
\input xy
\xyoption{all}

\newtheorem{lemma}{Lemma}[section]
\newtheorem{corollary}[lemma]{Corollary}
\newtheorem{theorem}[lemma]{Theorem}
\newtheorem{proposition}[lemma]{Proposition}

\theoremstyle{definition}
\newtheorem{remark}[lemma]{Remark}
\newtheorem{definition}[lemma]{Definition}

\DeclareMathOperator{\Mod}{Mod}
\DeclareMathOperator{\modd}{mod}
\DeclareMathOperator{\Endd}{End}
\DeclareMathOperator{\Add}{Add}
\DeclareMathOperator{\add}{add}
\DeclareMathOperator{\Hom}{Hom}
\DeclareMathOperator{\pd}{pd}
\DeclareMathOperator{\id}{id}
\DeclareMathOperator{\gl.dim}{gl.dim}
\DeclareMathOperator{\Ext}{Ext}
\DeclareMathOperator{\Lim}{Lim}
\DeclareMathOperator{\Sup}{Sup}
\DeclareMathOperator{\Id}{Id}

\newtheorem{question}[lemma]{Question}

\begin{document}

\title{pure semisimple $n$-cluster tilting subcategories}

\author{Ramin Ebrahimi}
\address{Department of Mathematics, University of Isfahan, P.O. Box: 81746-73441, Isfahan, Iran}
\email{ramin69@sci.ui.ac.ir}

\author{Alireza Nasr-Isfahani}
\address{Department of Mathematics, University of Isfahan, P.O. Box: 81746-73441, Isfahan, Iran\\
  and School of Mathematics, Institute for Research in Fundamental Sciences (IPM), P.O. Box: 19395-5746, Tehran, Iran}
\email{nasr$_{-}$a@sci.ui.ac.ir / nasr@ipm.ir}

\subjclass[2010]{{16E30}, {16G10}, {18E99}}

\keywords{$n$-cluster tilting subcategory, pure semisimple, $n$-homological pair, functor category}

\begin{abstract}
From the viewpoint of higher homological algebra, we introduce pure semisimple $n$-abelian categories, which are analogs of pure semisimple abelian categories. Let $\Lambda$ be an Artin algebra and $\mathcal{M}$ be an $n$-cluster tilting subcategory of $\Mod$-$\Lambda$. We show that $\mathcal{M}$ is pure semisimple if and only if each module in $\mathcal{M}$ is a direct sum of finitely generated modules. Let $\mathfrak{m}$ be an $n$-cluster tilting subcategory of $\modd$-$\Lambda$. We show that $\Add(\mathfrak{m})$ is an $n$-cluster tilting subcategory of $\Mod$-$\Lambda$ if and only if $\mathfrak{m}$ has an additive generator if and only if $\Mod(\mathfrak{m})$ is locally finite. This generalizes
Auslander's classical results on pure semisimplicity of Artin algebras.
\end{abstract}

\maketitle


\section{Introduction}
Higher Auslander-Reiten theory was introduced and developed by Iyama \cite{I1, I2}. It deals with $n$-cluster tilting subcategories of
abelian categories, where $n$ is a fixed positive integer. In this subcategories all short exact sequences are split, but there are nice exact sequences with $n+2$ objects. Recently, Jasso by modifying the axioms of abelian categories introduced $n$-abelian categories which are categories inhabited by certain exact sequences with $n+2$ terms, called $n$-exact sequences \cite{J}. $n$-abelian categories are an axiomatization of $n$-cluster tilting subcategories. Jasso shows that any $n$-cluster tilting subcategory of an abelian category is $n$-abelian. Furthermore, he also shows that $n$-abelian categories satisfying certain mild assumptions can be realized as $n$-cluster tilting subcategories of abelian categories. There is also a derived version of the theory focusing on $n$-cluster tilting subcategories of triangulated categories \cite{KR}. These categories were formalized to the theory of $(n+2)$-angulated categories by Geiss et al. \cite{GKO}.

Although there are rich examples of $n$-cluster tilting subcategories, constructing categories
having an $n$-cluster tilting subcategory is one of the main direction of research in this subject and it is a difficult task. On the other hand, since for the case $n=1$ we have ordinary abelian and triangulated categories, it is natural to ask which properties of abelian and triangulated categories can be generalized to the context of $n$-abelian and $n$-angulated categories. For examples of these directions see \cite{HI1, HI2, IO, Jo, HJV}.

Purity for Grothendieck categories was extensively studied by
Simson \cite{S1, S2}. Among other things he showed that a
Grothendieck category $\mathcal{A}$ is pure semisimple if and only
if each object of $\mathcal{A}$ is a direct sum of noetherian
subobjects. In particular case, for a left artinian ring $\Lambda$
the module category $\Mod$-$\Lambda$ is pure semisimple if and only
if each left $\Lambda$-module is isomorphic to a direct sum of
finitely generated modules. It is known that if $\Lambda$ is a
left artinian ring of finite representation type (i.e., there are
only finitely many non-isomorphic finitely generated
indecomposable $\Lambda$-modules) then $\Lambda$ is pure
semisimple. The converse of this fact is yet open, and is known as
the pure semisimplicity conjecture (see \cite{S4}, \cite{S5},
\cite{S3} and \cite{S6}). Nonetheless, Auslander in \cite{Au3}
proved that the pure semisimplicity conjecture is valid for Artin
algebras. He also showed that an Artin algebra $\Lambda$ is of
finite representation type if and only if the functor
$\Hom_{\Lambda}(S,-)$ is finite for each simple object $S$ in
$\modd$-$\Lambda$ if and only if the functor $\Hom_{\Lambda}(X,-)$ is
finite for each $X$ in $\modd$-$\Lambda$ if and only if
$\Mod(\modd$-$\Lambda)$ is locally finite, where $\Mod(\modd$-$\Lambda)$
is the category of all additive covariant functors from
$\modd$-$\Lambda$ to the category of abelian groups.

Let $\mathfrak{m}$ be an $n$-cluster tilting subcategory of $\modd$-$\Lambda$. Herschend et al. called the pair $(\Lambda,\mathfrak{m})$ an $n$-homological pair \cite{HJV}. In this paper we introduce the notion of pure semisimple $n$-homological pairs and prove important results about them (Corollary \ref{firstresult} and Theorem \ref{secondresult}). We say that $(\Lambda,\mathfrak{m})$ is a pure semisimple $n$-homological pair, provided that $\Add(\mathfrak{m})$ is an $n$-cluster tilting subcategory of $\Mod$-$\Lambda$. Section 2 is dedicated to explaining the origin of the name we gave to these objects. More precisely, in section 2, we define pure semisimple $n$-abelian categories and show that an $n$-cluster tilting subcategory $\mathcal{M}$ of $\Mod$-$\Lambda$ is pure semisimple if and only if each module in $\mathcal{M}$ is a direct sum of finitely generated modules. We say that an $n$-homological pair $(\Lambda,\mathfrak{m})$ is of finite type if $\mathfrak{m}$ has an additive generator. We show that an $n$-homological pair $(\Lambda,\mathfrak{m})$ is of finite type if and only if the functor $\Hom_{\mathfrak{m}}(X,-)$ is finite for each $X$ in $\mathfrak{m}$ if and only if $\Hom_{\Lambda}(S,-)$ is a finite object of $\Mod(\mathfrak{m})$ for each simple object $S$ in $\modd$-$\Lambda$ if and only if $\Mod(\mathfrak{m})$ is locally finite. The questions of finiteness and finite generation for $n$-cluster tilting subcategories, which are among the first that have been asked by Iyama \cite{I3}, are still open. Even the Iyama's question: "Dose there exists an $n$-cluster tilting subcategory of a category of finitely generated modules of an Artin algebra with $n\geq 2$ which has infinitely many isomorphism classes of indecomposables?" has no answer yet. In our main result we show that an $n$-homological pair $(\Lambda,\mathfrak{m})$ with $n\geq 2$ is pure semisimple if and only if $(\Lambda,\mathfrak{m})$ is of finite type. It shows that the Iyama's question is equivalent to the following question: Is any $n$-homological pair with $n \geq 2$ pure semisimple?

The paper is organized as follows. In section 2 we recall the definitions of $n$-abelian categories and $n$-cluster tilting subcategories and define purity for compactly generated $n$-abelian categories. Then we show that an $n$-cluster tilting subcategory $\mathcal{M}$ of $\Mod$-$\Lambda$ is pure semisimple if and only if every object of $\mathcal{M}$ is a direct sum of finitely generated objects. In section 3 we give a one direction of the main result, the corollary \ref{firstresult}, which shows that any $n$-homological pair of finite type is pure semisimple. Finally, in the last section we give another direction of the main result, the theorem \ref{secondresult}, which says that any pure semisimple $n$-homological pair is of finite type.

\subsection{Notation}
Throughout this paper $n$ always denotes a fixed positive integer. Let $\Lambda$ be an Artin algebra, we denote by $\Mod$-$\Lambda$ (resp., $\modd$-$\Lambda$) the category of all (resp., finitely generated) left $\Lambda$-modules. For a $\Lambda$-module $M$ we denote by $\pd(M)$ and $\id(M)$ its projective dimension and injective dimension, respectively. Also we denote by $\gl.dim(\Lambda)$ the global dimension of $\Lambda$. In this paper all categories are additive and subcategories are closed under direct summands. Let $\mathcal{C}$ be an additive category and $\mathcal{X}$ be a class of objects in $\mathcal{C}$. We denote by $\Add(\mathcal{X})$ (resp., $\add(\mathcal{X})$) the full subcategory of $\mathcal{C}$ whose objects are direct summands of (resp., finite) direct sums of objects in $\mathcal{X}$. For an additive category $\mathcal{C}$, we denote by $\mathcal{J}_{\mathcal{C}}$ the Jacobson radical of $\mathcal{C}$, where for each $X,Y\in \mathcal{C}$
$$\mathcal{J}_{\mathcal{C}}(X,Y)=\{h:X\rightarrow Y \mid 1_X-gh\;\text{is invertible for any}\;g:Y \rightarrow X \}.$$


\section{Pure semisimple $n$-abelian categories}
In this section we recall the definitions of $n$-abelian categories, $n$-cluster tilting subcategories and $n$-homological pairs. For further information and motivation of definitions the readers are referred to \cite{I1, I2, HJV, J}. Also we define compactly generated $n$-abelian categories and pure semisimplicity for these categories. For an Artin algebra $\Lambda$ we show that an $n$-cluster tilting subcategory $\mathcal{M}$ of $\Mod$-$\Lambda$ is a pure semisimple $n$-abelian category if and only if each objects of $\mathcal{M}$ is isomorphic to a direct sum of finitely generated objects of $\mathcal{M}$.

\subsection{n-Abelian categories}
Let $\mathcal{M}$ be an additive category and $d^0:X^0 \rightarrow X^1$ be a morphism in $\mathcal{M}$. An $n$-cokernel of $d^0$ is a sequence
\begin{equation}
(d^1, \ldots, d^n): X^1 \overset{d^1}{\rightarrow} X^2 \overset{d^2}{\rightarrow}\cdots \overset{d^{n-1}}{\rightarrow} X^n \overset{d^n}{\rightarrow} X^{n+1} \notag
\end{equation}
of objects and morphisms in $\mathcal{M}$ such that for all $Y\in \mathcal{M}$
the induced sequence of abelian groups
\begin{align}
0 \rightarrow \Hom(X^{n+1},Y) \rightarrow \Hom(X^n,Y) \rightarrow\cdots\rightarrow \Hom(X^1,Y) \rightarrow \Hom(X^0,Y) \notag
\end{align}
is exact \cite{J}. The concept of $n$-kernel of a morphism is defined dually.
\begin{definition}$($\cite[Definition 2.4]{J}$)$
Let $\mathcal{M}$ be an additive category. An $n$-exact sequence in $\mathcal{M}$ is a complex

\begin{equation}
X^0 \overset{d^0}{\rightarrow} X^1 \overset{d^1}{\rightarrow} X^2 \overset{d^2}{\rightarrow}\cdots \overset{d^{n-1}}{\rightarrow} X^n \overset{d^n}{\rightarrow} X^{n+1} \notag
\end{equation}
such that $(d^0, \ldots, d^{n-1})$ is an $n$-kernel of $d^n$ and $(d^1, \ldots, d^n)$ is an $n$-cokernel of $d^0$.
\end{definition}

\begin{definition}$($\cite[Definition 3.1]{J}$)$
Let $n$ be a positive integer. An $n$-abelian category is an additive category $\mathcal{M}$ which satisfies the following axioms:
\begin{itemize}
\item[(A0)]
The category $\mathcal{M}$ is idempotent complete.
\item[(A1)]
Every morphism in $\mathcal{M}$ has an $n$-kernel and an $n$-cokernel.
\item[(A2)]
For every monomorphism $d^0:X^0 \rightarrow X^1$ in $\mathcal{M}$ and for every $n$-cokernel $(d^1, \ldots, d^n)$ of $d^0$, the  following sequence is $n$-exact:
\begin{equation}
X^0 \overset{d^0}{\rightarrow} X^1 \overset{d^1}{\rightarrow} X^2 \overset{d^2}{\rightarrow}\cdots \overset{d^{n-1}}{\rightarrow} X^n \overset{d^n}{\rightarrow} X^{n+1} \notag
\end{equation}
\item[(A3)]
For every epimorphism $d^n:X^n \rightarrow X^{n+1}$ in $\mathcal{M}$ and for every $n$-kernel $(d^0, \ldots, d^{n-1})$ of $d^n$, the  following sequence is $n$-exact:
\begin{equation}
X^0 \overset{d^0}{\rightarrow} X^1 \overset{d^1}{\rightarrow} X^2 \overset{d^2}{\rightarrow}\cdots \overset{d^{n-1}}{\rightarrow} X^n \overset{d^n}{\rightarrow} X^{n+1} \notag
\end{equation}
\end{itemize}
\end{definition}

Motivated by the definition of compact objects in abelian categories \cite[Definition 18]{Mu}, we give the following definition.

\begin{definition}
Let $\mathcal{M}$ be an additive category with arbitrary direct sum. We call an object $X \in \mathcal{M}$ a compact object if any morphism from $X$ to a nonempty coproduct $\oplus_{i\in I} X_i$ factors through some
finite subcoproduct $\oplus_{i=1}^k X_i$. We say that $\mathcal{M}$ is compactly generated if for every $M\in \mathcal{M}$ there is an epimorphism $h:\oplus_{i\in I} X_i \rightarrow M$ where $X_i$ is compact for each $i\in I$.
\end{definition}

Now we are ready to define pure semisimple $n$-abelian categories.

\begin{definition}
Let $\mathcal{M}$ be a compactly generated $n$-abelian category.
\begin{itemize}
\item[(i)]
We say that an $n$-exact sequence
$$X^0 \rightarrow X^1 \rightarrow X^2 \rightarrow\cdots \rightarrow X^n \rightarrow X^{n+1}$$ is pure $n$-exact if for every compact object $Y,$ the induced sequence of abelian groups
$$
0 \rightarrow \Hom(Y, X^0) \rightarrow \Hom(Y, X^1) \rightarrow\cdots\rightarrow \Hom(Y, X^n) \rightarrow \Hom(Y, X^{n+1}) \rightarrow 0 \notag
$$
is exact. In this case we say that $X^0\rightarrow X^1$ is a pure monomorphism and $X^n\rightarrow X^{n+1}$ is a pure epimorphism. An object $P$ of $\mathcal{M}$ is called pure projective if for every pure $n$-exact sequence $X^0 \rightarrow X^1 \rightarrow X^2 \rightarrow\cdots \rightarrow X^n \rightarrow X^{n+1}$ the induced sequence of abelian groups
\begin{align}
0 \rightarrow \Hom(P, X^0) \rightarrow \Hom(P, X^1) \rightarrow\cdots\rightarrow \Hom(P, X^n) \rightarrow \Hom(P, X^{n+1}) \rightarrow 0 \notag
\end{align}
is exact.
\item[(ii)] We say that $\mathcal{M}$ is
pure semisimple if all objects of $\mathcal{M}$ are pure projective.
\end{itemize}
\end{definition}

\begin{remark}
Let $\mathcal{M}$ be a compactly generated $n$-abelian category and $P$ be an object of $\mathcal{M}$. It is easy to see that $P$ is pure projective if and only if for every pure epimorphism $f:X^n\rightarrow X^{n+1}$ and every morphism $g:P\rightarrow X^{n+1}$ there exists $\tilde{f}:P\rightarrow X^n$ such that the following diagram is commutative:
\begin{center}
\begin{tikzpicture}
\node (X1) at (1,1) {$P$};
\node (X2) at (1,-1) {$X^{n+1}$};
\node (X3) at (-1,-1) {$X^n$};
\draw [->,thick] (X1) -- (X2) node [midway,right] {$g$};
\draw [->,thick] (X3) -- (X2) node [midway,above] {$f$};
\draw [->,thick,dotted] (X1) -- (X3) node [midway,above] {$\tilde{f}$};
\end{tikzpicture}
\end{center}
\end{remark}

\subsection{Pure semisimplicity of $n$-cluster tilting subcategories }

In this subsection we first recall the definition of $n$-cluster tilting subcategories and then we give a characterization of pure semisimple $n$-cluster tilting subcategories.

Let $\mathcal{A}$ be an additive category. A subcategory $\mathcal{M}\subseteq \mathcal{A}$ is called contravariantly finite if for every $A\in \mathcal{A}$ there exist an object $M\in \mathcal{M}$ and a morphism $f:M \rightarrow A$ such that for each $N\in \mathcal{M}$ the sequence of abelian groups
$$\Hom(N,M) \rightarrow \Hom(N,A)\rightarrow 0$$
is exact. Such a morphism $f$ is called a right $\mathcal{M}$-approximation of $A$. The notion of covariantly finite subcategory and left $\mathcal{M}$-approximation is defined dually. A functorially finite subcategory of $\mathcal{A}$ is a subcategory which is both covariantly and contravariantly finite in $\mathcal{A}$ \cite{AS}.

Recall that a subcategory $\mathcal{M}$ of an abelian category $\mathcal{A}$ is called generating if for every object $X\in \mathcal{A}$ there exist an object $Y\in \mathcal{M}$ and an epimorphism $Y\rightarrow X$. The concept of cogenerating subcategory is defined dually.

\begin{definition}$($\cite[Definition 3.14]{J}$)$
Let $\mathcal{A}$ be an abelian category and $\mathcal{M}$ be a generating-cogenerating full subcategory of $\mathcal{A}$. $\mathcal{M}$ is called an $n$-cluster tilting subcategory of $\mathcal{A}$ if $\mathcal{M}$ is functorially finite in $\mathcal{A}$ and
\begin{align}
\mathcal{M}& = \{ X\in \mathcal{A} \mid \forall i\in \{1, \ldots, n-1 \}, \Ext^i(X,\mathcal{M})=0 \}\notag \\
                  & =\{ X\in \mathcal{A} \mid \forall i\in \{1, \ldots, n-1 \}, \Ext^i(\mathcal{M},X)=0 \}.\notag
\end{align}

Note that $\mathcal{A}$ itself is the unique 1-cluster tilting subcategory of $\mathcal{A}$.

\end{definition}
\begin{remark} {\rm Let $\mathcal{A}$ be an abelian category and $\mathcal{M}$ be an $n$-cluster tilting subcategory of $\mathcal{A}$. Since $\mathcal{M}$ is a generating-cogenerating subcategory of $\mathcal{A}$, for each $A\in \mathcal{A}$, every left $\mathcal{M}$-approximation of $A$ is a monomorphism and every right $\mathcal{M}$-approximation of $A$ is an epimorphism.}
\end{remark}

The following result gives a rich source of $n$-abelian categories.
\begin{theorem}$($\cite[Theorem 3.16]{J}$)$
Let $\mathcal{A}$ be an abelian category and $\mathcal{M}$ be an $n$-cluster tilting subcategory of $\mathcal{A}$. Then $\mathcal{M}$ is an $n$-abelian category.
\end{theorem}

\begin{lemma}
Let $\mathcal{M}$ be an $n$-cluster tilting subcategory of $\Mod$-$\Lambda$ and $\mathcal{M}_0$ be the subcategory of all compact objects in $\mathcal{M}$. Then we have $\mathcal{M}_0=\mathcal{M}\cap \modd$-$\Lambda$, and $\mathcal{M}$ is compactly generated.
\begin{proof}
First we note that $\mathcal{M}$ is closed under arbitrary direct sums, because the functor $\Ext^i(-,X)$ commute with direct sums.
It is obvious that $\mathcal{M}\cap \modd$-$\Lambda \subseteq \mathcal{M}_0$, for the converse inclusion let $M_0$ be a compact object in $\mathcal{M}$ which is not finitely generated. Since $M_0$ is not compact in $\Mod$-$\Lambda$, there exists a morphism $M_0 \rightarrow \oplus_{i\in I} X_i$ which dose not factor through a finite subcoproduct. If for each $i\in I$ we choose a left $\mathcal{M}$-approximation $X_i\rightarrow M_i$, then it is easy to see that the composition $M_0 \rightarrow \oplus_{i\in I} X_i\rightarrow \oplus_{i\in I} M_i$ dose not factor through a finite subcoproduct, which gives a contradiction. For the last part we note that the projective module $\Lambda\in \mathcal{M}_0$ and the assertion follows.
\end{proof}
\end{lemma}

Now we are ready to state the main theorem of this section.

\begin{theorem}\label{theorem1}
Let $\mathcal{M}$ be an $n$-cluster tilting subcategory of $\Mod$-$\Lambda$, then $\mathcal{M}$ is pure semisimple if and only if each module in $\mathcal{M}$ is a direct sum of finitely generated modules.
\end{theorem}
For the proof of the above theorem we need the following lemma.

\begin{lemma}\label{lemmax}
Let $\mathcal{M}$ be an $n$-cluster tilting subcategory of $\Mod$-$\Lambda$ such that each module in $\mathcal{M}$ has a finitely generated direct summand. Then each module in $\mathcal{M}$ is isomorphic to a direct sum of finitely generated modules.
\begin{proof}
Let $M_0=M\in \mathcal{M}$ be an arbitrary module. By assumption there is a finitely generated module $X_1$ and a module $M_1$ such that $M\simeq X_1\oplus M_1$. Inductively we can choose a family of finitely generated modules $\{X_i\}_{i\geq 1}$ and a family of modules $\{M_i\}_{i\geq 1}$ such that $M_i\simeq X_{i+1}\oplus M_{i+1}$ for each $i \geq 0$. We claim that $\oplus_{i\geq 1}X_i$ is a direct summand of $M$. By construction, $\oplus_{i=1}^k X_i$ is a direct summand of $M$. Consider a direct system $\{Y_k=\oplus_{i=1}^k X_i \}_{k\geq 1}$ with the obvious inclusion maps. It is clear that the direct system $\{Y_k \}_{k\geq 1}$ is a direct summand of the direct system $\{M\}_{i\geq 1}$ which all maps are identity. Since $\underrightarrow{\Lim}$ as a functor preserve section maps, $\underrightarrow{\Lim}Y_k\simeq \oplus_{i\geq 1}X_i$ is a direct summand of $M$. By the above argument and using Zorn's lemma we can choose a family of finitely generated modules $\{X_i\}_{i\in I}$ such that $\oplus_{i\in I}X_i$ is a direct summand of $M$ and it is maximal with this property. Thus there exists a module $N$ such that $M\simeq (\oplus_{i\in I}X_i)\oplus N$. By assumption $N$ has a finitely generated direct summand $N_0$. Therefore $(\oplus_{i\in I}X_i)\oplus N_0$ is a direct summand of $M$ which give a contradiction. Thus $M\simeq \oplus_{i\in I}X_i$.
\end{proof}
\end{lemma}
Now we are ready to prove the theorem \ref{theorem1}.

\begin{proof}[Proof of Theorem \ref{theorem1}]
By the lemma \ref{lemmax}, it is enough to show that every module in $\mathcal{M}$ has a finitely generated direct summand. Because $\mathcal{M}$ is closed under direct sums, for every module $M\in \mathcal{M}$ we have a pure epimorphism $f:\oplus_{i\in I}X_i \rightarrow M$, where $X_i$ is a finitely generated indecomposable $\Lambda$-module for each $i\in I$. Since $M$ is pure projective, $f$ is a retraction. So we have an split short exact sequence
$$0\rightarrow N\overset{\alpha}{\longrightarrow} \oplus_{i\in I}X_i \overset{f}{\longrightarrow} M\rightarrow 0.$$
Let $g=(g_i)_{i\in I}:M \rightarrow \oplus_{i\in I}X_i$ be a section of $f$ and $\beta:\oplus_{i\in I}X_i\rightarrow N$ be a retraction of $\alpha$. We claim that there exists $i\in I$ such that $f_i:X_i\rightarrow M$ is a section. Assume that for each $i\in I$, $f_i:X_i\rightarrow M$ is not a section. Since $\Endd_{\Lambda}(X_i)$ is a local ring, $g_if_i:X_i\rightarrow X_i\in\mathcal{J}(X_i,X_i)$. On the other hand we have that $1_{\oplus_{i\in I}X_i}=gf+\alpha \beta$. Thus $1_{X_i}=g_if_i+\alpha_i \beta_i$ for each $i\in I$. Since $g_if_i\in \mathcal{J}(X_i,X_i)$, $\alpha_i \beta_i$ is an isomorphism. Thus $\alpha \beta$ is an isomorphism which is a contradiction.
\end{proof}

\begin{definition}$($\cite[Definition 2.9]{HJV}$)$
Let $\Lambda$ be an Artin algebra and $\mathfrak{m}$ be an $n$-cluster tilting subcategory of $\modd$-$\Lambda$. Then $(\Lambda,\mathfrak{m})$ is called an $n$-homological pair.
\end{definition}

\begin{definition}
We say that an $n$-homological pair $(\Lambda,\mathfrak{m})$ is pure semisimple if $\Add(\mathfrak{m})$ is an $n$-cluster tilting subcategory of $\Mod$-$\Lambda$. Also we say that an $n$-homological pair $(\Lambda,\mathfrak{m})$ is of finite type if $\mathfrak{m}$ has an additive generator.
\end{definition}

\section{$n$-homological pairs of finite type are pure semisimple}

In this section we show that if $(\Lambda,\mathfrak{m})$ be an $n$-homological pair of finite type, then $\Add(\mathfrak{m})$ is an $n$-cluster tilting subcategory of $\Mod$-$\Lambda$ which is pure semisimple by the theorem \ref{theorem1}. This shows that any $n$-homological pair of finite type is pure semisimple.

We recall some well known results that we will need in the rest of the paper.

\begin{lemma}\label{lemma2}
Let $\Lambda$ be an Artin algebra, then
$$\Sup\{\pd(X)\mid X\in \Mod-\Lambda\}=\Sup\{\pd(X)\mid X\in \modd-\Lambda\}$$
\begin{proof}
See for example the theorem 4.1.2 of \cite{We}.
\end{proof}
\end{lemma}

\begin{lemma}\label{lemma4}$($\cite[Proposition 2.4.1]{I2}$)$
Let $\mathcal{A}$ be an abelian category with enough projectives and injectives and $\mathcal{P}$ be the full subcategory of all projectives. A functorially finite subcategory $\mathcal{C}\subseteq \mathcal{A}$ is $n$-cluster tilting subcategory if and only if $\mathcal{P}\subseteq \mathcal{C}$ and
$$\mathcal{C}=\{X\in \mathcal{A}\mid \forall i\in \{1, \ldots, n-1 \}, \Ext^i(\mathcal{M},X)=0 \}.$$
\end{lemma}

Now we can prove the main result of this section.

\begin{theorem}\label{result1}
Let $M$ be a finitely generated left $\Lambda$-module. $\Add(M)$ is an $n$-cluster tilting subcategory of $\Mod$-$\Lambda$ if and only if  $\add(M)$ is an $n$-cluster tilting subcategory of $\modd$-$\Lambda$.
\begin{proof} Without loss of generality, we can assume that $M$ is a basic finitely generated left $\Lambda$-module. If $\Add(M)$ is an $n$-cluster tilting subcategory of $\Mod$-$\Lambda$, then it is easy to see that $\add(M)$ is an $n$-cluster tilting subcategory of $\modd$-$\Lambda$. Now assume that $\add(M)$ is an $n$-cluster tilting subcategory of $\modd$-$\Lambda$. Since $M$ is $n$-rigid by assumption and $\Add(M)$ is functorially finite in $\Mod$-$\Lambda$, by the lemma \ref{lemma4}, it is enough to show that if $X\in \Mod$-$\Lambda$ such that $\Ext^i_{\Lambda}(M,X)=0$ for each $i\in\{1, 2, \ldots, n-1\}$, then $X\in \Add(M)$. Let $$0\rightarrow X\rightarrow I^0 \rightarrow I^1 \rightarrow\cdots\rightarrow I^n$$
be an injective resolution of $X$. Since $\Ext^i_{\Lambda}(M,X)=0$ for each $i\in\{1, 2, \ldots, n-1\}$, we have an exact sequence
$$0\rightarrow \Hom_{\Lambda}(M,X)\rightarrow \Hom_{\Lambda}(M,I^0) \rightarrow \Hom_{\Lambda}(M,I^1) \rightarrow \cdots \rightarrow \Hom_{\Lambda}(M,I^n) \rightarrow C\rightarrow 0.$$
Let $\Gamma=\Endd_\Lambda(M)^{op}$. Since $\gl.dim(\Gamma) \leq n+1$, $\Hom_{\Lambda}(M,X)$ is a projective $\Gamma$-module. We know that any projective module over an Artin algebra is a direct sum of indecomposable projective modules. Indecomposable projective $\Gamma$-modules are of the form $\Hom_{\Lambda}(M,M_i)$ such that $M_i$ is an indecomposable direct summand of $M$. Therefore
$$\Hom_{\Lambda}(M,X)\simeq \oplus_{i\in I}\Hom_{\Lambda}(M,M_i)\simeq \Hom_{\Lambda}(M,\oplus_{i\in I}M_i).$$
Since
\begin{align}
\Hom_{\Gamma}(\Hom_{\Lambda}(M,\oplus_{i\in I}M_i),\Hom_{\Lambda}(M,X))&\simeq \Hom_{\Gamma}(\oplus_{i\in I}\Hom_{\Lambda}(M,M_i),\Hom_{\Lambda}(M,X)) \notag \\
&\simeq \prod_{i\in I} \Hom_{\Gamma}(\Hom_{\Lambda}(M, M_i),\Hom_{\Lambda}(M, X)) \notag \\
&\simeq \prod_{i\in I} \Hom_{\Lambda}(M_i, X) \notag \\
&\simeq \Hom_{\Lambda}(\oplus_{i\in I} M_i, X), \notag
\end{align}
there exists a morphism $h:\oplus_{i\in I}M_i \rightarrow X$ such that
$\Hom_{\Lambda}(M,h):\Hom_{\Lambda}(M,\oplus_{i\in I}M_i)\rightarrow \Hom_{\Lambda}(M,X)$ is the above isomorphism. We show that $h:\oplus_{i\in I}M_i \rightarrow X$ is an isomorphism. First consider the exact sequence
$$0\rightarrow K\rightarrow \oplus_{i\in I}M_i \overset{h}{\rightarrow} X,$$
where $K$ is the kernel of $h$. Applying the functor $\Hom_{\Lambda}(M,-)$ we conclude that $K=0$ because $M$ is a generating module and so $h$ is a monomorphism. Now applying the functor $\Hom_{\Lambda}(M,-)$ to the exact sequence
$$0\rightarrow \oplus_{i\in I}M_i \overset{h}{\rightarrow} X\overset{g}\rightarrow C \rightarrow0,$$
where $C$ is the cokernel of $h$. We get an exact sequence
$$0\rightarrow \Hom_{\Lambda}(M,\oplus_{i\in I}M_i)\overset{\Hom_{\Lambda}(M, h)}\longrightarrow \Hom_{\Lambda}(M, X)\overset{\Hom_{\Lambda}(M, g)} \longrightarrow \Hom_{\Lambda}(M, C).$$
Since $\Hom_{\Lambda}(M, h)$ is an isomorphism, $\Hom_{\Lambda}(M, g)=0$. $M$ is a generating module and so $C=0$. Therefore $h$ is an isomorphism and the result follows.
\end{proof}
\end{theorem}

\begin{corollary}\label{firstresult}
Let $(\Lambda,\mathfrak{m})$ be an $n$-homological pair of finite type, then $(\Lambda,\mathfrak{m})$ is pure semisimple.
\begin{proof}
Let $M$ be an additive generator of $\mathfrak{m}$. Then by the theorem \ref{result1}, $\Add(\mathfrak{m})=\Add(M)$ is an $n$-cluster tilting subcategory of $\Mod$-$\Lambda$ and the result follows.
\end{proof}
\end{corollary}

The following corollary is an immediate consequence of the $n$-Auslander correspondence \cite{I2} (see also \cite{I3}) and the theorem \ref{result1}.

\begin{corollary}
There are bijections between the set of equivalence classes of
$n$-cluster tilting subcategories $\mathcal{M}$ with additive
generators of $\modd$-$\Lambda$ for Artin algebras $\Lambda$, the
set of isomorphism classes of basic finitely generated left
$\Lambda$-modules $M$ that $\Add(M)$ are $n$-cluster tilting
subcategories of $\Mod$-$\Lambda$ for Artin algebras $\Lambda$ and
the set of Morita-equivalence classes of $n$-Auslander algebras.
\end{corollary}

Motivated by the theorem \ref{result1} and corollary \ref{firstresult} we pose the following question.

\begin{question} \label{ques1}Let $\Lambda$ be an Artin algebra and $\mathfrak{m}$ be an $n$-cluster tilting subcategory of $\modd$-$\Lambda$.
\begin{itemize}
\item[(i)]
 Is there an $n$-cluster tilting subcategory $\mathcal{M}$ of $\Mod$-$\Lambda$ that contains $\mathfrak{m}$?

\item[(ii)]
Assume that there exists an $n$-cluster tilting subcategory $\mathcal{M}$ of $\Mod$-$\Lambda$ that contains $\mathfrak{m}$. When we can describe the objects of $\mathcal{M}$ in terms of the objects of $\mathfrak{m}$?
\end{itemize}
\end{question}


\section{pure semisimple $n$-homological pairs are of finite type}

In this section we show that if an $n$-homological pair $(\Lambda,\mathfrak{m})$ is pure semisimple then $(\Lambda,\mathfrak{m})$ is of finite type.

\subsection{The functor category}

In this subsection we recall some preliminaries on functor categories. For further information the reader is referred to \cite{Au, Au1}.

Let $(\Lambda,\mathfrak{m})$ be an $n$-homological pair, we denote by $\Mod(\mathfrak{m})$ the category of all additive covariant functors from $\mathfrak{m}$ to the category of abelian groups. Objects of $\Mod(\mathfrak{m})$ are called $\mathfrak{m}$-modules and for $\mathfrak{m}$-modules $F_1$ and $F_2$ we denote by $\Hom_{\mathfrak{m}}(F_1,F_2)$ the set of all natural transformations from $F_1$ to $F_2$. It is known that $\Mod(\mathfrak{m})$ is an abelian category. Kernels, cokernels, product, direct sum and exactness are all defined pointwise. For each $M\in \modd$-$\Lambda$ we denote the functor $\Hom_{\Lambda}(M,-):\mathfrak{m}\longrightarrow Ab$ by $(M,-)$. It is well known that for each $M\in \mathfrak{m}$, $(M,-)$ is a projective object in $\Mod(\mathfrak{m})$. For every $F\in \Mod(\mathfrak{m})$ there exists an exact sequence

$$\oplus_{i\in I}(Y_i,-)\rightarrow \oplus_{j\in J}(X_j,-)\rightarrow F\rightarrow 0, $$
where $X_i$ and $Y_j$ are in $\mathfrak{m}$ for each $i, j$. We recall that $F$ is said to be finitely generated if the set $J$ can be chosen to be finite, and $F$ is said to be finitely presented if both the sets $I$ and $J$ can be chosen to be finite. In the other words, $F$ is finitely generated if and only if there is an epimorphism $(X,-)\rightarrow F\rightarrow 0$ with $X\in \mathfrak{m}$, and is finitely presented if and only if there is an exact sequence $(Y,-)\rightarrow (X,-)\rightarrow 0$ with $X,Y\in \mathfrak{m}$. Because $\mathfrak{m}$ is idempotent complete the Yoneda functor $P:\mathfrak{m}\rightarrow \Mod(\mathfrak{m})$ induces an equivalence $P:\mathfrak{m}\rightarrow \mathfrak{p}(\mathfrak{m})$ where $\mathfrak{p}(\mathfrak{m})$ is the category of all finitely generated projective $\mathfrak{m}$-modules.

Following \cite{Au2} we say that an $\mathfrak{m}$-module $F$ is noetherian (resp., artinian) if it satisfies the ascending (resp., descending) chain condition on submodules. We say that an $\mathfrak{m}$-module $F$ is finite if it is both noetherian and artinian. A functor $F$ is called simple if it is not zero and the only subfunctors of $F$ are $0$ and $F$.

\begin{definition}$($\cite{Au2}$)$
An $\mathfrak{m}$-module $F$ is said to be locally finite if every finitely generated submodule of $F$ is finite. The category $\Mod(\mathfrak{m})$ is said to be locally finite if every $\mathfrak{m}$-module is locally finite.
\end{definition}

\begin{proposition}$($\cite[Proposition 1.11]{Au2}$)$\label{prop1}
Let $(\Lambda,\mathfrak{m})$ be an $n$-homological pair. Then the following statements are equivalent:
\begin{itemize}
\item[a)]
$\Mod(\mathfrak{m})$ is locally finite.
\item[b)]
Every finitely generated $\mathfrak{m}$-module is finite.
\item[c)]
$(X,-)$ is finite for each $X\in \mathfrak{m}$.
\item[d)]
Every simple $\mathfrak{m}$-module is finitely presented and every nonzero $\mathfrak{m}$-module has a simple submodule.
\end{itemize}
\end{proposition}

We use the following description of simple modules from \cite{Au2}. Since an indecomposable object $X\in \mathfrak{m}$ has a local endomorphism ring, the indecomposable projective object $(X, -)$ has a unique maximal subfunctor denoted by $\mathcal{J}(X, -)$. Thus for any indecomposable object $X\in \mathfrak{m}$, the functor $\frac{(X, -)}{\mathcal{J}(X, -)}$ is simple. Moreover, given any simple functor $F\in \Mod(\mathfrak{m})$, there is a unique (up to isomorphism) indecomposable object $X\in \mathfrak{m}$ such that $\frac{(X, -)}{\mathcal{J}(X, -)}\simeq F$. Hence the correspondence $X\mapsto \frac{(X, -)}{\mathcal{J}(X, -)}$ gives a bijection between the isomorphism classes of simple objects in $\Mod(\mathfrak{m})$ and the isomorphism classes of indecomposable objects in $\mathfrak{m}$.

Let $X$ be an indecomposable object in $\mathfrak{m}$. We recall that a morphism $f:X\rightarrow Y$ is called left almost split if
\begin{itemize}
\item[a)]
$f$ is not a section.
\item[b)]
If $g:X\rightarrow Z$ is a morphism in $\mathfrak{m}$ which is not a section, then there is a morphism $h:Y\rightarrow Z$ such that $hf=g$.
\end{itemize}

\begin{lemma}\label{lemmaa}$($\cite[Corollary 2.6]{Au2}$)$
Let $(\Lambda,\mathfrak{m})$ be an $n$-homological pair and $X$ be an indecomposable object in $\mathfrak{m}$. The simple $\mathfrak{m}$-module $\frac{(X,-)}{\mathcal{J}(X,-)}$ is finitely presented if and only if there is a left almost split morphism $f:X\rightarrow Y$. Further, if $f:X\rightarrow Y$ is a left almost split morphism, then $(Y,-)\rightarrow (X,-)\rightarrow \frac{(X,-)}{\mathcal{J}(X,-)} \rightarrow 0$ is exact and is a finite projective presentation of $\frac{(X,-)}{\mathcal{J}(X,-)}$.
\end{lemma}

We recall that if $X$ be an indecomposable object in $\mathfrak{m}$, then there exists a left almost split morphism $f:X\rightarrow Y$ (see \cite[Section 3.3.1]{I1}).

We now use the description of the simple $\mathfrak{m}$-modules to describe when a nonzero $\mathfrak{m}$-module has a simple submodule.
\begin{definition}$($\cite{Au2}$)$
Let $F$ be an $\mathfrak{m}$-module and $X$ be an object in $\mathfrak{m}$ (not necessarily indecomposable). An element $x$ in $F(X)$ is said to be universally minimal if $x\neq 0$ and has the property that given any morphism $g:X\rightarrow Y$ in $\mathfrak{m}$ which is not a section, then $F(g)(x)=0$.
\end{definition}

\begin{proposition}\label{theoremb}$($\cite[Proposition 2.9]{Au2}$)$
Let $F$ be an $\mathfrak{m}$-module and $X\in \mathfrak{m}.$
\begin{itemize}
\item[a)]
An element $x$ in $F(X)$ is universally minimal if and only if $X$ is indecomposable in $\mathfrak{m}$ and the morphism $(X,-)\rightarrow F$ corresponding to $x$ has a simple image.
\item[b)]
$F$ has a simple submodule if and only if $F(X)$ has a universally minimal element for some $X\in \mathfrak{m}$.
\end{itemize}
\end{proposition}

\subsection{The main theorem}
In this subsection we prove the main theorem of this section. We begin with the following easy remark.
\begin{remark}\label{remarkc}
Let $(\Lambda,\mathfrak{m})$ be an $n$-homological pair. If $I$ is a directed set and $\{X_i,\varphi_j^i\}_{i\leq j}$ is a direct system in $\mathfrak{m}$ over $I$, since any $Y\in \mathfrak{m}$ is finitely generated we have the following functorial isomorphism of abelian groups
$$\Hom_{\Lambda}(Y,\underrightarrow{\Lim}X_i)\simeq\underrightarrow{\Lim}\Hom_{\Lambda}(Y,X_i).$$
\end{remark}

If $\mathfrak{m}$ is a pure semisimple $n$-cluster tilting subcategory of $\modd$-$\Lambda$, then $\mathcal{M}=\Add(\mathfrak{m})$ is an $n$-cluster tilting subcategory of $\Mod$-$\Lambda$. Let $F:\mathfrak{m}\rightarrow Ab$ be an additive functor, then obviously  we can extend $F$ to the additive functor, also denote by $F$, from $\mathcal{M}$ to the category of abelian groups.

The following lemma is adapted from \cite[Page 5]{Au3}. We give the proof for the convenience of the reader.

\begin{lemma}\label{lemmad}
Let $(\Lambda,\mathfrak{m})$ be an $n$-homological pair, $I$ be a directed set, $\{X_i,\varphi_j^i\}_{i\leq j}$ be a direct system in $\mathfrak{m}$ over $I$ and $F:\mathfrak{m}\rightarrow Ab$ be an $\mathfrak{m}$-module. Then we have a functorial isomorphism
$$F(\underrightarrow{\Lim}X_i)\simeq \underrightarrow{\Lim}F(X_i).$$
\begin{proof}
If we consider $F$ as a functor from $\mathcal{M}=\Add(\mathfrak{m})$ to the category of abelian groups, then $F$ has a projective resolution
$$\oplus_{s\in S}(M_s,-)\rightarrow \oplus_{t\in T}(N_t,-)\rightarrow F\rightarrow 0.$$
By the remark \ref{remarkc}, we have a functorial isomorphism
$$\oplus_{t\in T}(N_t,\underrightarrow{\Lim}X_i)\simeq \oplus_{t\in T}(\underrightarrow{\Lim}(N_t,X_i))\simeq \underrightarrow{\Lim}(\oplus_{t\in T}(N_t,X_i))$$
Thus we have a commutative exact diagram
\begin{center}
\begin{tikzpicture}
\node (X1) at (-4,3) {$0$};
\node (X2) at (0,3) {$0$};
\node (X3) at (-4,1) {$\oplus_{s\in S}(M_s,\underrightarrow{\Lim}X_i)$};
\node (X4) at (0,1) {$\oplus_{t\in T}(N_t,\underrightarrow{\Lim}X_i)$};
\node (X5) at (4,1) {$F(\underrightarrow{\Lim}X_i)$};
\node (X6) at (7,1) {$0$};
\node (X7) at (-4,-1) {$\underrightarrow{\Lim}(\oplus_{s\in S}(M_s,X_i))$};
\node (X8) at (0,-1) {$\underrightarrow{\Lim}(\oplus_{t\in T}(N_t,X_i))$};
\node (X9) at (4,-1) {$\underrightarrow{\Lim}F(X_i)$};
\node (X10) at (7,-1) {$0$};
\node (X11) at (-4,-3) {$0$};
\node (X12) at (0,-3) {$0$};
\draw [->,thick] (X1) -- (X3) node [midway,right] {};
\draw [->,thick] (X2) -- (X4) node [midway,above] {};
\draw [->,thick] (X3) -- (X4) node [midway,above] {};
\draw [->,thick] (X4) -- (X5) node [midway,above] {};
\draw [->,thick] (X5) -- (X6) node [midway,above] {};
\draw [->,thick] (X3) -- (X7) node [midway,above] {};
\draw [->,thick] (X4) -- (X8) node [midway,above] {};
\draw [->,thick] (X5) -- (X9) node [midway,above] {};
\draw [->,thick] (X7) -- (X8) node [midway,above] {};
\draw [->,thick] (X8) -- (X9) node [midway,above] {};
\draw [->,thick] (X9) -- (X10) node [midway,above] {};
\draw [->,thick] (X7) -- (X11) node [midway,above] {};
\draw [->,thick] (X8) -- (X12) node [midway,above] {};
\end{tikzpicture}
\end{center}
Hence the right-hand vertical morphism is an isomorphism, which proves the lemma.
\end{proof}
\end{lemma}

We need the following well known technical lemma.

\begin{lemma}\label{lemmae}$($\cite[Lemma 5.30]{Ro}$)$
Let $R$ be an arbitrary ring, $\{Y_i,\varphi_j^i\}_{i\leq j}$ be a direct system of left $R$-modules over a directed set $I$ and $\lambda_i:Y_i\rightarrow \underrightarrow{\Lim}Y_i$ be morphisms in the construction of direct limit. For any $y_i\in Y_i$ we have that $\lambda_i(y_i)=0$ if and only if $\varphi_t^i(y_i)=0$ for some $i\leq t$.
\end{lemma}

The following lemma is the key step for proving the main theorem of this section.

\begin{lemma}\label{keylemma}
Let $(\Lambda,\mathfrak{m})$ be a pure semisimple $n$-homological pair. If $F:\mathfrak{m}\rightarrow Ab$ is a nonzero additive functor, then $F$ has a simple subfunctor.
\begin{proof}
By the proposition \ref{theoremb} it is enough to show that there is an object $Z$ in $\mathfrak{m}$ and a universally minimal element $z\in F(Z)$. Since $F$ is a nonzero functor, there is an indecomposable object $X$ in $\mathfrak{m}$ such that $F(X)\neq 0$. Choose a nonzero element $x\in F(X)$. We show that there is a morphism $h:X\rightarrow Z$ such that $F(h)(x)$ is universally minimal in $F(Z)$.
Consider the following set
$$\Omega=\{f:X\rightarrow Y \mid \text{$Y$ is an indecomposable object in $\mathfrak{m}$ and $F(f)(x)\neq 0$} \}.$$
Let $f_1:X\rightarrow Y_1$ and $f_2:X\rightarrow Y_2$ be two elements of $\Omega$. We define the following relation in $\Omega$:
$$f_1\preceq f_2 \Longleftrightarrow \exists g:Y_1\rightarrow Y_2 \;\text{such that}\; f_2=gf_1.$$
It is easy to check that $\preceq$ is a partial order relation. We show that $\Omega$ satisfies the assumptions of the Zorn's lemma. First $\Omega \neq \emptyset$ because $\Id_X \in \Omega$. Now assume that $\{f_i:X\rightarrow Y_i \}_{i\in I}$ is a chain in $\Omega$. Put $Y=\underrightarrow{\Lim}Y_i$. Since $\mathfrak{m}$ is pure semisimple, there is a family of indecomposable objects $\{Z_j \}_{j\in J}$ in $\mathfrak{m}$ such that $Y\simeq \oplus_{j\in J}Z_j$. By the lemma \ref{lemmad}, $F$ commute with direct limit and especially with direct sum. Then we have an isomorphism
$$F(Y)\simeq \underrightarrow{\Lim}_{i\in I}F(Y_i)\simeq \oplus_{j\in J} F(Z_j).$$
Consider the following direct limit diagram
\begin{center}
\begin{tikzpicture}
\node (X1) at (-2,0) {$X$};
\node (X2) at (0,1) {$Y_s$};
\node (X3) at (0,-1) {$Y_t$};
\node (X4) at (3,0) {$\underrightarrow{\Lim}_{i\in I}Y_i\simeq \oplus_{j\in J} Z_j$};
\draw [->,thick] (X1) -- (X2) node [midway,above] {$f_s$};
\draw [->,thick] (X1) -- (X3) node [midway,below] {$f_t$};
\draw [->,thick] (X2) -- (X3) node [midway,right] {$f_{s,t}$};
\draw [->,thick] (X2) -- (X4) node [midway,above] {$\lambda_s$};
\draw [->,thick] (X3) -- (X4) node [midway,below] {$\lambda_t$};
\end{tikzpicture}
\end{center}
Applying $F$ we have a direct limit diagram
\begin{center}
\begin{tikzpicture}
\node (X1) at (-4,0) {$F(X)$};
\node (X2) at (0,2) {$F(Y_s)$};
\node (X3) at (0,-2) {$F(Y_t)$};
\node (X4) at (4,0) {$\oplus_{j\in J} F(Z_j)$};
\draw [->,thick] (X1) -- (X2) node [midway,above] {$F(f_s)$};
\draw [->,thick] (X1) -- (X3) node [midway,below] {$F(f_t)$};
\draw [->,thick] (X2) -- (X3) node [midway,right] {$F(f_{s,t})$};
\draw [->,thick] (X2) -- (X4) node [midway,above] {$F(\lambda_s)$};
\draw [->,thick] (X3) -- (X4) node [midway,below] {$F(\lambda_t)$};
\end{tikzpicture}
\end{center}
For every $i\in I$, set $y_i=F(f_i)(x)\in F(Y_i)$. By properties of direct limit for every $s,t\in I$ we know that $F(\lambda_s)(y_s)=F(\lambda_t)(y_t)$. Put $z=F(\lambda_s)(y_s)=F(\lambda_t)(y_t)$. By the lemma \ref{lemmae}, $z$ is a nonzero element of $\oplus_{j\in J} F(Z_j)$. Thus there is at least one $j\in J$ such that $F(p_j)(z)\neq 0$, where $p_j:\oplus_{j\in J} F(Z_j)\rightarrow Z_j$ is the canonical projection. Now we set $h=p_jg:X\rightarrow Z_j$, where $g=\lambda_sf_s$ for some $s\in I$. It is easy to check that $F(h)(x)\neq 0$, and $h:X\rightarrow Z_j$ is an upper bound for the chain $\{f_i:X\rightarrow Y_i \}_{i\in I}$. Thus $\Omega$ satisfies the assumptions of the Zorn's lemma. We choose the maximal element $f:X\rightarrow Z$ of $\Omega$, then $F(f)(x)$ is a universally minimal element in $F(Z)$.
\end{proof}
\end{lemma}

For the proof of the next theorem we need the following lemma.

\begin{lemma}\label{lemmaz}
Let $(\Lambda,\mathfrak{m})$ be an $n$-homological pair. Then
\begin{itemize}
\item[a)]
If $0\rightarrow M_1\rightarrow M_2\rightarrow M_3 \rightarrow 0$ is an exact sequence of $\mathfrak{m}$-modules with $M_2$ locally finite, then $M_1$ and $M_3$ are both locally finite.
\item[b)]
An $\mathfrak{m}$-module $F$ is finite if and only if for each indecomposable object $X\in \mathfrak{m}$, $F(X)$ is a finite $\Endd_{\Lambda}(X)$-module, and $F(X)=0$ for all but a finite number of indecomposables $X\in \mathfrak{m}$.
\end{itemize}
\begin{proof}
See the proposition 1.9 and the theorem 2.12 of \cite{Au2}.
\end{proof}
\end{lemma}

The following theorem is a higher dimensional analogue of the theorem 3.1 of \cite{Au2}. Note that for the technical reasons we work with the covariant functors instead of the contravariant functors.

\begin{theorem}\label{theoremf}
Let $(\Lambda,\mathfrak{m})$ be an $n$-homological pair. The following statements are equivalent.
\begin{itemize}
\item[a)]
$\Mod(\mathfrak{m})$ is locally finite.
\item[b)]
$(X,-)$ is finite for each $X$ in $\mathfrak{m}$.
\item[c)]
$(S,-)$ is a finite object of $\Mod(\mathfrak{m})$ for each simple object $S$ in $\modd$-$\Lambda$.
\item[d)]
$(\Lambda,\mathfrak{m})$ is of finite type.
\end{itemize}
\begin{proof}
a) $\Longrightarrow$ b). Follows by the proposition \ref{prop1}.

b) $\Longrightarrow$ c). Let $f:S\rightarrow M$ be a left $\mathfrak{m}$-approximation of $S$. Thus we have an exact sequence $(M,-)\rightarrow (S,-)\rightarrow 0$. Since $(M,-)$ is finite, $(S,-)$ is also finite by the lemma \ref{lemmaz}.

c) $\Longrightarrow$ d). Let $\{S_1, \ldots, S_t\}$ be a complete set of non-isomorphic simple $\Lambda$-modules. Because each nonzero $\Lambda$-module has a simple submodule, we know that for any $X$ in $\mathfrak{m}$, $(\oplus_{i=1}^tS_i,X)=0$ implies that $X=0$. In particular, $(\oplus_{i=1}^tS_i,X)\neq 0$ for each indecomposable object $X$ in $\mathfrak{m}$. Since each $(S_i,-)$ is a finite $\mathfrak{m}$-module, it follows that $\oplus_{i=1}^t(S_i,-)=(\oplus_{i=1}^tS_i,-)$ is a finite $\mathfrak{m}$-module. Thus by the lemma \ref{lemmaz}, there is only a finite number $X_1, \ldots, X_k$ of non-isomorphic indecomposable objects in $\mathfrak{m}$ such that $(\oplus_{i=1}^tS_i,X)\neq 0$. Therefore $\{X_1, \ldots, X_k\}$ is a complete set of non-isomorphic indecomposable objects in $\mathfrak{m}$ and $(\Lambda,\mathfrak{m})$ is of finite type.

d) $\Longrightarrow$ a). Since $(\Lambda,\mathfrak{m})$ is of finite type, by the corollary \ref{firstresult}, $(\Lambda,\mathfrak{m})$ is pure semisimple. Thus by the lemma \ref{keylemma} each nonzero functor $F:\mathfrak{m}\rightarrow Ab$ has a simple subfunctor. By \cite[3.3.1]{I1}, $\mathfrak{m}$ has left almost split morphisms and so by the lemma \ref{lemmaa} each simple functor in $\Mod(\mathfrak{m})$ is finitely presented. Therefore $\Mod(\mathfrak{m})$ is locally finite by the proposition \ref{prop1}.
\end{proof}
\end{theorem}

Now we can prove the main theorem of this section.

\begin{theorem}\label{secondresult}
An $n$-homological pair $(\Lambda,\mathfrak{m})$ is pure semisimple if and only if $(\Lambda,\mathfrak{m})$ is of finite type.
\begin{proof}
The necessary condition follows from the corollary \ref{firstresult}. Now assume that $(\Lambda,\mathfrak{m})$ is pure semisimple. Since each simple $\mathfrak{m}$-module is finitely presented and by the lemma \ref{keylemma} any nonzero $\mathfrak{m}$-module has a simple subfunctor, by the proposition \ref{prop1}, $\Mod(\mathfrak{m})$ is locally finite. Then the result follows by the theorem \ref{theoremf}.
\end{proof}
\end{theorem}

\begin{remark}
Iyama in \cite{I3} asked the following question:
\begin{itemize}
\item[]
Does any $n$-cluster tilting subcategory of $\modd$-$\Lambda$ with $n \geq 2$ have an additive generator?
\end{itemize}
By the theorem \ref{secondresult}, Iyama's question equivalent to the following question:
\begin{itemize}
\item[]
Is any $n$-homological pair with $n \geq 2$ pure semisimple?
\end{itemize}
It is obvious that the positive answer of this question will answer positively the question \ref{ques1}$(i)$.
\end{remark}

Recall that the first Brauer-Thrall conjecture asserts that any
Artin algebra is either representation-finite or there exist
indecomposable modules with arbitrarily large length. Roiter
proved the first Brauer-Thrall conjecture for finite dimensional
algebras \cite{Roi} (see also \cite{R}). Auslander proved the conjecture for Artin algebras using
Auslander-Reiten theory and the Harada-Sai lemma \cite{Au2}.

We say that an $n$-homological pair $(\Lambda,\mathfrak{m})$ is of bounded length if the lengths of the finitely generated indecomposable left $\Lambda$-modules which are contained in $\mathfrak{m}$ are bounded.

The following theorem is a higher dimensional analogue of the first Brauer-Thrall conjecture. The proof of the following theorem is an easy adaptation of the proof of the first Brauer-Thrall conjecture (see section 2.3 of \cite{R}), so we omit the proof.

\begin{theorem}\label{secondresult1}
An $n$-homological pair $(\Lambda,\mathfrak{m})$ is of finite type if and only if $(\Lambda,\mathfrak{m})$ is of bounded length.
\end{theorem}

Now we summarize our results in the following corollary.

\begin{corollary} Let $(\Lambda,\mathfrak{m})$ be an $n$-homological pair. The following statements are equivalent.
\begin{itemize}
\item[1)]
$(\Lambda,\mathfrak{m})$ is pure semisimple.
\item[2)]
$(\Lambda,\mathfrak{m})$ is of finite type.
\item[3)]
$(\Lambda,\mathfrak{m})$ is of bounded length.
\item[4)]
$\Mod(\mathfrak{m})$ is locally finite.
\item[5)]
$(X,-)$ is finite for each $X$ in $\mathfrak{m}$.
\item[6)]
$(S,-)$ is a finite object of $\Mod(\mathfrak{m})$ for each simple object $S$ in $\modd$-$\Lambda$.

\end{itemize}
\end{corollary}

\section*{acknowledgements}
The authors would like to thank the referee for a careful reading of this paper and making helpful suggestions that improved the presentation of the paper. Also we would like to thank Daniel Simson for his comments on the earlier version of this paper. The research of the second
author was in part supported by a grant from IPM (No. 98170412).

\end{document}